\documentclass[a4paper,11pt, reqno]{amsart}
\usepackage{amsfonts, amsmath, amssymb, amsgen, amsthm, amscd,
latexsym}
\usepackage{color}
\usepackage[all]{xy}

\def\Id{\mathop{\rm Id}\nolimits}

\def\Id{\mathop{\rm Id}\nolimits}

\def\b{\beta}
\def\d{\delta}
\def\D{\Delta}

\def\s{\sigma}

\def\ve{\varepsilon}
\def\vp{\varphi}

\def\ot{\otimes}
\def\ol{\overline}
\def\ra{\rightarrow}

\def\lt{\triangleleft}

\def\0D{\Delta^{(0)}}
\def\1D{\Delta^{(1)}}

\def\Co{\,\square\,}

\newtheorem{theorem}{Theorem}[section]
\newtheorem{remark}[theorem]{Remark}

\newtheorem{lemma}[theorem]{Lemma}

\newtheorem{example}[theorem]{Example}
\newtheorem{definition}[theorem]{Definition}

\def\build#1_#2^#3{\mathrel{
\mathop{\kern 0pt#1}\limits_{#2}^{#3}}}
\newcommand{\ns}[1]{~\hspace{-4pt}_{_{{<#1>}}}}
\newcommand{\ps}[1]{~\hspace{-4pt}^{^{(#1)}}}

\numberwithin{equation}{section}

 \newcommand{\ie}{{\it i.e.\/}\ }

\def\b{\beta}

\def\d{\delta}

\def\s{\sigma}

\def\ve{\varepsilon}

\def\vp{\varphi}

\def\D{\Delta}

\def\ot{\otimes}
\def\part{\partial}

\def\ra{\rightarrow}

\def\text{\hbox}

\def\ot{\otimes}

\def\ra{\rightarrow}

\def\wh{\widehat}

\def\Id{\mathop{\rm Id}\nolimits}

\def\build#1_#2^#3{\mathrel{
\mathop{\kern 0pt#1}\limits_{#2}^{#3}}}

\numberwithin{equation}{section}
\newcommand{\comment}[1]{\relax}

\numberwithin{equation}{section}

\begin{document}
\title{\bf  Hopf Galois (Co)extensions In Noncommutative Geometry }

\author{Mohammad Hassanzadeh}

\address{Institut des Hautes \'Etudes Scientifiques, IHES, Bures Sur Yvette, France.\\
}
\email{mhassanz@ihes.fr}
\maketitle

\begin{abstract}
We introduce an alternative proof, with the use of tools and notions for Hopf algebras, to show  that Hopf Galois coextensions of coalgebras are the sources of stable anti Yetter-Drinfeld modules. Furthermore we show that two natural cohomology theories related to a Hopf Galois coextension are isomorphic.
\end{abstract}

\section*{Introduction}

The general definition of  Hopf Galois extensions   was introduced by Kreimer and Takeuchi \cite{kt} around the same time when noncommutative geometry \cite{NCG} started to develop with works of Alain Connes in $1980$. In fact Hopf Galois extensions are noncommutative analogue of  affine torsors  and principal
bundles. Descent data related to a Hopf Galois extension has its roots in works of Grothendieck in algebraic geometry \cite{gr}.
Basically there are two well-known homology theories which are related to a Hopf Galois extension $A(B)_H$. One is the relative cyclic homology of the algebra extension  $B\subseteq A$, and another one is the Hopf cyclic homology of the Hopf algebra $H$ involved in the Hopf Galois extension. It is shown  in \cite{JS}  that these two homology theories are isomorphic where the Hopf cyclic homology of Hopf algebra $H$ has  coefficients in a stable anti Yetter-Drinfeld (SAYD) module. Not only this isomorphism is given by the canonical isomorphism of the Hopf Galois extension, but also the SAYD coefficients, which is a module-comodule over $H$,  is constructed by  this canonical map.
A great idea here is that the cyclic (co)homology of (co)algebra (co)extensions which is not easy to compute  is isomorphic to the Hopf cyclic (co)homology of the Hopf algebra of the (co)extension  which can be computed more easily.
The dual notion of Hopf Galois coextensions is introduced by Schneider in \cite{schneider} and  can be viewed as a noncommutative
generalization of the theory of quotients of formal schemes under free actions of formal group schemes \cite{schneider2}.

In this paper   we recall  the basics of Hopf Galois (co)extensions. Also we study the module-comodule  and specially stable anti Yetter-Drinfeld module structures over the Hopf algebra involved in a Hopf Galois (co)extension. Furthermore  we study the related (co)homology theories to a Hopf Galois (co)extension and finally we observe that these (co)homology theories are isomorphic. More precisely, in Section $1$  we recall the results of Jara and Stefan in \cite{JS} for Hopf Galois extensions and   in Section $2$  we introduce an alternative proof for the dual case of Hopf Galois coextensions using the tools and notions for Hopf algebras.
Although the main results of the second section of this paper have been already proved for a general case of Equivaraint Hopf Galois coextensions for $\times$-Hopf coalgebras in \cite{hr2},  a direct proof for the case of Hopf algebras using  the related concepts and notions, which has not been appeared  in the literature, has its own advantages.

\medskip
\textbf{Acknowledgments}: The author would like to thank the  Institut des Hautes \'Etudes Scientifiques, IHES, for its hospitality and financial supports during his visit when the whole work was accomplished.

\medskip

\textbf{Notations}: In this paper we denote a Hopf algebra by $H$ and  its counit   by $\ve$. We use the Sweedler summation notation $\Delta(h)= h\ps{1}\ot h\ps{2}$ for the coproduct of a Hopf algebra. Furthermore the summation notations $\blacktriangledown(h)= h\ns{-1}\ot h\ns{0}$ and $\blacktriangledown(h)= h\ns{0}\ot h\ns{1}$  are used for the left and right coactions of a coalgebra, respectively.

\tableofcontents

\section{Hopf cyclic cohomology}
In this section we review the basics of (co)cyclic modules. We explain cyclic duality  to obtain a cyclic module from a cocyclic module and vice versa. As an example we study the Hopf cyclic (co)homology with coefficients in a stable anti Yetter-Drinfeld module.
\subsection{Cyclic modules and cyclic duality}
     A cosimplicial module \cite{NCG}, \cite{Loday}   contains  $\mathbb{C}$-modules   $C^{n}$, $n\geq 0$, with  $\mathbb{C}$-module maps    $\delta_{i}:C^{n}\longrightarrow C^{n+1}$ called cofaces, and  $\sigma_{i}:C^{n}\longrightarrow C^{n-1}$ called codegeneracies  satisfying the following cosimplicial relations;
\begin{eqnarray}
\begin{split}
& \delta_{j}  \delta_{i} = \delta_{i} \delta_{j-1},     ~~\text{ if} \quad\quad i <j,\\
& \sigma_{j}  \sigma_{i} = \sigma_{i} \sigma_{j+1},     \text{    if} \quad\quad i \leq j,\\
&\sigma_{j} \delta_{i} =   \label{rel1}
 \begin{cases}
\delta_{i} \sigma_{j-1},   \quad
 &\text{if} \quad\text{$i<j$,}\\
\text{Id},   \quad\quad
 &\text{if} \quad   \text{$i=j$ or $i=j+1$,}\\
\delta_{i-1} \sigma_{j},  \quad
 &\text{if} \quad \text{$i>j+1$}.
\end{cases}
\end{split}
\end{eqnarray}

A cocyclic module  is a cosimplicial module  with extra morphisms  $\tau :C^n\longrightarrow C^n$ which are called cocyclic maps such that the following extra commutativity relations hold.
\begin{eqnarray}
\begin{split}
&\tau\delta_{i}=\delta_{i-1} \tau, \hspace{43 pt} 1\le i\le  n+1,\\
&\tau \delta_{0} = \delta_{n+1}, \hspace{43 pt} \tau \sigma_{i} = \sigma _{i-1} \tau,  \hspace{33 pt} 1\le i\le n,\\ \label{rel2}
&\tau \sigma_{0} = \sigma_{n} \tau^2, \hspace{43 pt} \tau^{n+1} = \Id.
\end{split}
\end{eqnarray}

 Dually a cyclic module is given by quadruple   $ C=(C_{n},\delta_i, \sigma_{i}, \tau_{n})$  where $C_{n}$'s, $n\geq 0$, are  $\mathbb{C}$-modules and there are  $\mathbb{C}$-module maps $\delta_{i}: C_{n}\longrightarrow C_{n-1},$ called faces, $ \sigma_{i}:C_n \ra C_{n+1},  \quad 0 \leq i \leq n$ called degeneracies and  $\tau:C_n \ra C_n$  called cyclic maps satisfying the following  commutativity relations;
\begin{eqnarray}
\begin{split}
 &  \delta_{i} \delta_{j}=  \delta_{j-1} \delta_{i},  ~~~~ \text{if} \quad i <j,\\
& \sigma_{i}  \sigma_{j} = \sigma_{j+1} \sigma_{i},   ~~\text{if} \quad i \leq j,\\
&\delta_{i} \sigma_{j} = \label{rel11}
 \left\{\begin{matrix}
\sigma_{j-1} \delta_{i},   \hspace{10 pt}\text{if} \hspace{20 pt} i<j,\\
\Id,  \hspace{22 pt}  \text{if}   \hspace{22 pt} i=j ~~ \text{or}~~ &i=j+1,\\
\sigma_{j} \delta_{i-1},  \hspace{8 pt}\text{if} \hspace{8 pt} i>j+1,
\end{matrix}\right.
\end{split}
\end{eqnarray}
and;
\begin{align}
\begin{split}
&\delta_{i}\tau= \tau\delta_{i-1}, \hspace{43 pt} 1\le i\le n ,\\
& \delta_{0} \tau= \delta_{n} , \quad \sigma_{i} \tau= \tau  \sigma _{i-1},\hspace{23 pt} 1\le i\le n ,\\
&\sigma_{0} \tau = \tau^2  \sigma_{n} , \quad\tau^{n+1} = \Id. \\ \label{rel22}
\end{split}
\end{align}

Now we recall the duality procedure for (co)cyclic modules \cite{NCG}. For any cyclic module $C= (C_{n},\d_i, \sigma_i, \tau)$ we define its cyclic dual  by $\breve{C}^{n}= C_n $  with the following  cofaces, codegeneracies and cyclic morphisms.
\begin{align}
\begin{split}
& d_0: \tau_n \sigma_{n-1}, \quad d_i:= \sigma_{i-1}: \breve{C}^{n}\longrightarrow \breve{C}^{n+1}, \quad\quad 1\leq i \leq n,\\
&s_i:= \delta_i: \breve{C}^{n}\longrightarrow \breve{C}^{n-1}, \quad\quad 0\leq i \leq n-1, \\
& t:= \tau^{-1}.
\end{split}
\end{align}
Conversely for any cocyclic module  $C=(C^{n}, d_i, s_i, t)$ one obtains its cyclic dual denoted by $\widetilde{C}$ where;
\begin{align}
\begin{split}
&\d_i:=s_i: \widetilde{C}_n\longrightarrow \widetilde{C}_{n-1}, \quad 0\leq i\leq n-1,\quad \d_n:= \d_0 \tau_n,\\
&\sigma_i:=d_{i+1}:\widetilde{C}_n\longrightarrow \widetilde{C}_{n+1}, \quad 0\le i\le n-1,\\
&\tau:= t^{-1}.
\end{split}
\end{align}

\subsection{Cyclic cohomology of Hopf algebras}
In this subsection, we study some examples of (co)cyclic modules and their dual  theories for Hopf algebras which will be used later in Sections $2$ and $3$.
Hopf cyclic cohomology was introduced by Connes and Moscovici in \cite{CM98} and was generalized to Hopf cyclic cohomology with coefficients in a stable anti Yetter-Drinfeld module in \cite{HaKhRaSo2}.
\medskip

A left-right  anti Yetter-Drinfeld (AYD) module  $M$ over a Hopf algebra $H$ is a left module and a right comodule over $H$ satisfying the following compatibility condition \cite{HaKhRaSo1};
\begin{equation}
  (h\triangleright m)\ns{0}\ot (h\triangleright m)\ns{1}=h\ps{2}\triangleright m\ns{0}\ot h\ps{3}m\ns{1} S(h\ps{1}).
\end{equation}
This is called stable if $m\ns{1}\triangleright m\ns{0}=m$ for all $m\in M$. We use left-right of SAYD modules in Section $2$ for Hopf Galois extensions.
Similarly  a right-left SAYD module $M$ over $H$ is a right module and a left comodule satisfying the following compatibility condition,
\begin{equation}
  (m\triangleleft h)\ns{-1} \ot (m\triangleleft h)\ns{0}= S(h\ps{3})m\ns{-1}h\ps{1}\ot m\ns{0}\triangleleft  h\ps{2}.
\end{equation}
This type of SAYD modules will be used in Section $3$ for Hopf Galois coextensions.
A generalization of Connes-Moscovici cocyclic module for Hopf algebras with coefficients in a SAYD module was introduced in \cite{HaKhRaSo2}.
The following cyclic module is the dual cyclic module of the mentioned cocyclic module for a left-right SAYD module $M$ over $H$ where $C_n(H,M)=H^{\ot n+1}\ot_H M$.
\begin{align}\label{dual-Hopf-cyclic-module}
 &\delta_i(h_0\ot \cdots h_n\ot_H m)=h_0\ot \cdots \ot \ve(h_i)\ot \cdots h_n\ot_H m, \nonumber \\
 & \sigma_i(h_0\ot \cdots h_n\ot_H m)=h_0\ot \cdots \ot \Delta(h_i)\ot \cdots \ot h_n\ot_H m, \nonumber\\
 &\tau_n(h_0\ot \cdots h_n\ot_H m)=h_n m\ns{1}\ot h_0\ot \cdots \ot h_{n-1}\ot_H m\ns{0}.
\end{align}
We use this cyclic module in Section $2$ for Hopf Galois extensions.
The authors in \cite{bm1} introduced a new cyclic module for Hopf algebras which is   a dual of Connes-Moscovici cocyclic module for Hopf algebras in some sense. Later they have shown  in \cite{bm2} that this cyclic module is isomorphic to the cyclic dual of Connes-Moscovici cocyclic module for Hopf algebras. A generalization of this cyclic module with coefficients in a SAYD module was introduced in \cite{HaKhRaSo2}. The following cocyclic module is the dual cocyclic module of the mentioned cyclic module for a right-left SAYD module $M$ over Hopf algebra $H$ where  $C^n(H,M)= H^{\ot n}\ot M.$
\begin{align}\label{dual-Hopf-cocyclic-module}
\begin{split}
 &\delta_{i}(\widetilde{h}\ot m)= h_1\ot \cdots \ot h_i \ot 1_H \ot \cdots\ot h_n \ot m,\\
 &\delta_{n}(\widetilde{b}\ot m)= h_1\ps{1}\ot\cdots \ot h_n\ps{1}\ot S(h_1\ps{2}\cdots h_n\ps{2})m\ns{-1} \ot m\ns{0},\\
 &\s_{i}(\widetilde{h}\ot m)= h_1\ot \cdots \ot h_i  h_{i+1}\ot  \cdots\ot h_n \ot m,\\
 &\s_{n}(\widetilde{h}\ot m)=h_1\ot \cdots \ot h_{n-1}\ve(h_n)\ot m , \\
 &\tau_{n}(\widetilde{h}\ot m)= h_2\ps{1}\ot\cdots\ot h_n\ps{1}\ot S(h_1\ps{2} \cdots h_n\ps{2})m\ns{-1}\ot m\ns{0}\triangleleft h_1\ps{1}.
 \end{split}
\end{align}
Here $\widetilde{h}= h_1\ot\cdots \ot h_n$. We use this cocyclic module in Section $3$ for Hopf Galois coextensions. For more about Hopf cyclic (co)homology we refer the reader to \cite{CM98}, \cite{HaKhRaSo2}, \cite{masoud-survey} and \cite{kay}.

\section{Hopf Galois extensions of algebras}

\subsection{Preliminaries of Hopf Galois extensions}
\label{section-1}

In this section we recall the notion of Hopf Galois extensions. We review  the results in \cite{JS} which imply  that Hopf Galois extensions are  sources of producing  stable anti Yetter-Drinfeld modules. Furthermore  we study  the relation between the relative cyclic homology of the algebra extension and the Hopf cyclic homology of the Hopf algebra involved in a Hopf Galois extension.

\medskip

Let $H$ be a Hopf algebra and $A$ a right $H$-comodule algebra with the coaction $\rho: A\longrightarrow A\ot H$. The set of coinvariants of this coaction $B=\{a\in A, \quad \rho(a)=a\ot 1_H\}$ is a subalgebra of $A$. The algebra extension $B\subseteq A$ is Hopf Galois if the canonical map
\begin{equation}
  \b: A\ot_B A\longrightarrow A\ot H, \quad a\ot_B a'\longmapsto a a'\ns{0}\ot a'\ns{1},
\end{equation}
is bijective. We denote such a Hopf Galois extension by $A(B)^H$.

The map $\b$ is an isomorphism of left $A$-modules and right $H$-comodules where the left $A$-module structures of $A\ot_B A$ and $A\ot H$ are given by
$a(a_1\ot_B a_2)= aa_1\ot_B a_2$ and $a(a'\ot h)= aa'\ot h$,  respectively,  and also the right $H$ comodule structures of $A\ot_B A$ and $A\ot H$ are given by $a\ot a'\longmapsto a\ot a'\ns{0}\ot a'\ns{1}$ and $a\ot h\longmapsto a\ot h\ps{1}\ot h\ps{2}$, respectively, for all $a,a',a_1,a_2\in A$ and $h\in H$. In fact an extension $B\subseteq A$ is Hopf Galois if $A\ot_B A$ and $A\ot H$ are isomorphic in $_A\mathcal{M}^H$, the category of left $A$-modules and right $H$-comodules, by the canonical map $\beta$.
The first natural question is why we call this extension Galois and what is its relation with classical Galois extension of fields.

\begin{example}
  {\rm \emph{Classical Galois extensions of fields}:\\
Let $  F\subseteq E$  be  a field extension with  the Galois group $G$. This extension is Galois if and only if $G$ acts faithfully on $E$. This is equivalent to
$\mid[E:F]\mid = \mid G \mid$. Let $\mid G \mid=n$ and $G=\{x_1,\cdots , x_n \}$. Furthermore suppose $\{b_1,\cdots, b_n\}$ be the basis of $E/F$ and $\{p^1,\cdots, p^n\}$ be the dual basis of $\{x_i\}_i$ in $(kG)^*$ where $kG$ is the group algebra of $G$. Since $G$ is finite then $(kG)^*$ is a Hopf algebra  and $p^n(x_i)=\delta_{ij}$. Furthermore  the action of $G$ and therefore $kG$  on $E$ amounts to a coaction of  $(kG)^*$  on $E$ as follows;
\begin{equation}
  E\longrightarrow E\ot (kG)^*, \quad a\longmapsto x_i\triangleright a\ot p^i.
\end{equation}
One checks that $E$ is a right $(kG)^*$-comodule algebra by this coaction. We define the canonical  Galois map to be;
\begin{equation}
  \b: E\ot_F E\longrightarrow E\ot_k (kG)^*, \quad a\ot b\longmapsto a(x_i\triangleright b)\ot p^i.
\end{equation}
One uses the independence of $\{p^i\}$ to show that $\b$ is injective. Furthermore since both tensor products are finite-dimensional $F$-algebras then $\b$ is a surjection.
  }
\end{example}

\begin{example}\label{trivial}{\rm
     A  Hopf algebra $H$ over the field  $F$ is a right $H$-comodule algebra where the right coaction is given by the comultiplication of $H$. For  $A=H$ we have  $B=F 1_H$. Therefore $H(F)_H$ is a $H$-Galois extension with the canonical map which is given by;
  \begin{equation}
  \beta: H\ot_F H\longrightarrow H\ot_F H, \quad h\ot k\longmapsto hk\ps{1}\ot k\ps{2}, \quad h,k\in H,
  \end{equation}
  with the inverse map  $\b^{-1}(h\ot k)= hS(k\ps{1})\ot k\ps{2}$.

  }
\end{example}
The following example shows that Hopf Galois extensions are algebraic analogue of principal bundles.
\begin{example}
{\rm \emph{Principal bundles}:\\
Let $P(M,G)$ denote the principal bundle of the smooth manifold $P$ over the base $M$ with the structure Lie group $G$. There is a smooth right action of $G$ on $P$ denoted by $\triangleleft: P\times G\longrightarrow P$ which is free, \ie
\begin{equation}
  u\triangleleft g= u\triangleleft g'\Longrightarrow g=g', \quad g,g'\in G, u\in P.
\end{equation}
If $G$ is finite then freeness of the $G$-action is equivalent to the injectivity of the following map;
\begin{equation}
  \beta: P\times G\longrightarrow P\times P, \quad (u,g)\longmapsto (u,u\triangleleft g).
\end{equation}
One notes that the orbit space is isomorphic to the base space $P/G\cong M$ and the canonical projection $\pi: P\longrightarrow M$ is smooth. Now we apply a noncommutative approach to the map $\b$ and we consider the space of functions on the principal bundle. Let $A=C^{\infty}(P)$, $B=C^{\infty}(M)$ and $H=(kG)^*$. We dualize the action $\triangleleft$ to obtain a right coaction $\blacktriangledown: A\longrightarrow A\ot (kG)^*$. Therefore the following induced map
\begin{equation}
  (m_A\ot\Id)\circ (\Id\ot \blacktriangledown): A\ot A\longrightarrow A\ot H,
\end{equation}
is surjective where $m_A$ denotes the multiplication of $A$. Now if we restrict the tensor product  in the domain of the previous map to the coinvariant space of the coaction $\blacktriangledown$, we obtain an isomorphism. This example is a motivation  to define the notion of quantum principal bundles. A right $H$-comodule algebra $A$ is called a quantum principal bundle if  the Galois   map $\b$  related to the coaction  of $H$ is an isomorphism. For more about this example we refer the reader to \cite{haj}, \cite{dur},  \cite{bm} and \cite{kz}.

}

\end{example}

\begin{remark}{\rm
There are two major difference between  Hopf Galois theory and the classical Galois theory of field extensions. Classical Galois extensions can be characterized by the normal, separable field extensions without explicitly mentioning the Galois group. Furthermore, a Galois field extension determines uniquely the Galois group. There is no similar result for Hopf Galois theory, although a  characterization exists for Hopf Galois extensions of Hopf algebroids with some finiteness conditions.   Another difference
is that the fundamental theorem of Galois theory of fields extensions  does not hold for the
Hopf Galois extensions of algebras. We refer the reader for more in this regard to \cite{bal2} and \cite{joost}.
}

\end{remark}
Hopf  Galois extensions   were generalized later in different ways. For example the authors in \cite{bh} have introduced the notion of coalgebra extensions. Also see \cite{SS}. Furthermore,  Galois extensions have been studied for extended versions of Hopf algebras such as  Hopf algebroids in \cite{b1} and $\times$-Hopf algebras in \cite{bs2} and \cite{hr}.
\subsection{Homology theories related to  Hopf Galois extensions}

 Basically there are two well-known homology theories which are related to a Hopf Galois extension $A(B)_H$. One is the relative cyclic homology of the algebra extension  $B\subseteq A$ and another one is the Hopf cyclic homology of the Hopf algebra $H$ involved in the Hopf Galois extension. In this subsection we recall that these two homology theories are isomorphic. Here the Hopf cyclic homology of $H$ has  coefficients in a SAYD module which is constructed by the canonical isomorphism of the Hopf Galois extension. This shows that Hopf Galois extensions are the  sources of  SAYD modules. In this subsection we recall some related  results from \cite{JS} to be able to compare them with the similar results in the dual case of Hopf Galois coextensions in Section $3$.

\medskip

Suppose $B$ be an associative algebra over the field of complex numbers. Let $M$ and $N$ be $B$-bimodules. The cyclic tensor product  of $M$ and $N$ is defined to be $M\widehat{\ot}_BN:= (M\ot_B N)\ot_{B^e} B$, \cite{Q}. It can be shown that
\begin{equation}
  M\widehat{\ot}_B N\cong \frac{ M\ot_B N}{[M\ot_B N, B]},
\end{equation}
where the bracket stands for the subspace generated by all commutators. In fact,
\begin{equation}
  M\widehat{\ot}_B N= \frac{M\ot_B N}{\thicksim},
\end{equation}
where $\thicksim$ is the  equivalence relation defined by the following relation;
\begin{equation}
  \{ bm\ot_B n = m\ot_B nb\}.
\end{equation}
Similarly if $M_1, \cdots, M_n$ are $B$-bimodules then their cyclic tensor product can be defined as
\begin{equation}
  M_1\widehat{\ot}\cdots \widehat{\ot}M_n:= (M_1\ot_B \cdots \ot M_n)\ot_{B^e} B.
\end{equation}
Now let $B\subseteq A$ be  an algebra extension and $M$ a $B$-bimodule. Let $C_n(A(B), M)= M\widehat{\ot}_B A^{\widehat{\ot}_B} $. One defines a  simplicial module where the faces are given by

\begin{equation}
d_i(m \widehat{\ot}_B a_1 \widehat{\ot}_B \ldots \widehat{\ot}_B a_n )=
\begin{cases}
ma_1\widehat{\ot}_B a_2\widehat{\ot}_B\cdots \widehat{\ot}_Ba_n, & \text{$i=0$},\\
m\widehat{\ot}_B a_1\widehat{\ot}\cdots \widehat{\ot}_B a_i a_{i+1} \widehat{\ot}_B \cdots \widehat{\ot}_B a_n ,& \text{$0 < i < n$},\\
a_nm\widehat{\ot}_B a_1\widehat{\ot}_B \cdots \widehat{\ot}_B a_{n-1}, & \text{$i=n$},
\end{cases}
\end{equation}
and the degeneracies are defined as follows;
\begin{equation}
  s_i(m \widehat{\ot}_B a_1 \widehat{\ot}_B \ldots \widehat{\ot}_B a_n)=m \widehat{\ot}_B a_1 \widehat{\ot}_B \cdots \widehat{\ot}_B 1_A \widehat{\ot}_B \cdots \widehat{\ot}_B a_n.
\end{equation}
It is shown in \cite{JS} that if $M=A$, then the following cyclic operator turns $C_n(A(B), A)$ in to a cyclic module;
\begin{equation}
  t_n(a_0 \widehat{\ot}_B a_1 \widehat{\ot}_B \ldots \widehat{\ot}_B a_n)=a_n \widehat{\ot}_B a_0 \widehat{\ot}_B \ldots \widehat{\ot}_B a_{n-1}.
\end{equation}
The cyclic homology of this cyclic module is called relative cyclic homology of the algebra extension  $B\subseteq A$ and it is denoted by $HC_*(A(B),A)$.
Although the space $A\ot_B A$ has not a well-defined algebra structure, the subspace $(A\ot_B A)^B$ is an associative algebra by the following multiplication;
\begin{equation}\label{algebrast}
  (a_1\ot_B a'_1)(a_2\ot_B a'_2)= a_1 a_2\ot_B a'_2a_2.
\end{equation}
The canonical map $\beta$ induces the following isomorphism;
$$\overline{\b}: (A\ot_B A)^B\longrightarrow A^B\ot H.$$
Therefore in spit of the fact that the canonical isomorphism of the Hopf Galois extension is not an algebra map, the  algebra structure  \eqref{algebrast} enables us to obtain the following anti-algebra map;
\begin{equation}
  \kappa: H\longrightarrow (A\ot_B A)^B,\quad \kappa:= \overline{\b}^{-1}\circ i,
\end{equation}
where $i:H\longrightarrow A^B\ot H$ is given by $h\longmapsto 1_A \ot h$. We denote the summation notation $\kappa(h)= \kappa^1(h)\ot \kappa^2(h)$ for the image of the map $\kappa$. Now let $M$ be a $A$-bimodule. We set;
\begin{equation}
  M^B=\{m\in M, \quad bm=mb, \quad \forall b\in B\}, \quad M_B= \frac{M}{[M,B]}.
\end{equation}
The map $\kappa$ enables us to define a right $H$-module structure on $M^B$ given by
\begin{equation}
  mh=\kappa^1(h)m \kappa^2(h).
\end{equation}
Furthermore a left $H$-module structure on $M_B$ can be defined by
\begin{equation}\label{action-sayd}
  hm= \kappa^2(h) m \kappa^1(h).
\end{equation}
In a special case when $M=A$, the quotient space $A_B$ is also a right $H$-comodule by the original coaction of $H$ over $A$. This coaction with the action defined in \eqref{action-sayd} amounts $A_B$ to a left-right SAYD module over $H$.

For any Hopf Galois extension $A(B)_H$, one iteratively use the map $\b$ to transfer the cyclic structure of $C_n(A(B), A)$  to $C_n(H,A_B)=H^{\ot n}\ot_H H\ot A_B$. The result cyclic module on $C_n(H,A_B)$ is proved in \cite{JS} to be the cyclic module associated to Hopf cyclic homology of Hopf algebra $H$ with coefficients in SAYD module $A_B$ which is introduced in \eqref{dual-Hopf-cyclic-module}. The results of this subsection have been later generalized for Hopf Galois extensions of $\times$-Hopf algebras in \cite{bs2} and for the Equivariant Hopf Galois extensions of $\times$-Hopf algebras in \cite{hr}.
We summarize the headlines of this section  to be able to compare them with the dual case of Hopf Galois coextensions in the sequel section.
\begin{remark}\label{remark}{\rm
For any Hopf Galois extension $A(B)^H$, we have;
\begin{enumerate}
 \item [i)] The relative cyclic module of the algebra extension is a quotient space given by,
 $$B\ot_{B^e} A\ot_B\ot_B \cdots \ot_B A= A\widehat{\ot}_B\cdots \widehat{\ot}_B A= \frac{A\ot_B\cdots \ot_B A}{\sim}.$$
 \item[ii)] The objects $A\ot_B A$ and $A\ot H$ are isomorphic in the category of $_A\mathcal{M}^H$.
 \item [iii)] The subspace $A^B$ is an subalgebra of $A$.
 \item[iv)] The subspace $(A\ot_B A)^B$ is an algebra.
 \item [v)] The subspace $A^B$ is a left $(A\ot_B A)^B$-module and a right $H$-module.
  \item [vi)] The quotient space  $A_B$ is a right $(A\ot_B A)^B$-module and a left $H$-module.
 \item [vii)]The quotient space $A_B\cong B\ot_{B^e} A$ is a left-right SAYD module over $H$.
 \item [viii)] $HC_*(H,A_B)\cong HC_*(A(B),A)$.
\end{enumerate}

}

\end{remark}
\section{Hopf Galois coextensions of coalgebras}
In this section we use the notions and tools for Hopf algebras and Hopf cyclic cohomology to prove that  Hopf Galois coextensions of coalgebras are the sources of stable anti Yetter-Drinfeld modules. Furthermore we  show that the  Hopf cyclic cohomology of the Hopf algebra involved in a Hopf Galois coextension with coefficients in a SAYD  module, which is dual to $A_B$,  is isomorphic to the relative cyclic cohomology of  coalgebra coextension. These are the  dual of the results in \cite{JS} for Hopf Galois extensions which are explained in Subsection $2.2$.
The authors in \cite{hr2} have proved similar  results  for $\times$-Hopf coalgebras.
\subsection{Preliminaries of Hopf Galois coextensions}
Let $H$ be a Hopf algebra and $C$ a right $H$-module coalgebra with the action $ \triangleleft: C \ot H\longrightarrow C$. The set
\begin{equation}
  I= \{ c \triangleleft h - \ve(h)c\},
\end{equation}
is a two-sided coideal of $C$ and therefore $D=\frac{C}{I}$ is a coalgebra. The natural surjection $\pi: C\twoheadrightarrow D$ is called a right $H$-Galois coextension if the canonical map
\begin{equation}\label{galois}
  \beta: C\ot H\longrightarrow C\Co_D C, \quad ~~~ c\ot h\longmapsto c\ps{1}\ot c\ps{2}\triangleleft h,
\end{equation}
is a bijection. Such a Hopf Galois coextensions is denoted by $C(D)^H$. Here the $D$-bicomodule structure of $C$ is given by
\begin{equation}
  c\longmapsto \pi(c\ps{1})\ot c\ps{2}, \quad ~~~~~~  c\longmapsto c\ps{1}\ot \pi(c\ps{2}).
\end{equation}

Since by the definition of $D$ for all $h\in H$ and $c\in C$ we have
\begin{equation}\label{pi}
  \pi(c\triangleleft h)= \ve(h)\pi(c),
\end{equation}
the map $\b$ defined in \eqref{galois} is well-defined.
We denote the inverse of Galois map  $\beta$ by the following summation notation;
\begin{equation}
\beta^{-1}(c\Co_D c'):= \beta_{-}(c\Co_D c')\ot \beta_{+}(c\Co_D c'),
\end{equation}
where $Im\beta_{-}\in C$ and $Im\beta_{+}\in H.$ If there is no confusion  we can simply write $\beta=\beta_{-} \ot \beta_{+}$.
Since $C\Co_D C$ is not a coalgebra, the map $\b$ is not an isomorphism of coalgebras . Instead it is an isomorphism of left $C$-comodules and  right $H$-modules. The map $\b$ is a  left $C$-comodule  map where the left $C$-comodule structures of $C\Co_D C$ and $C\ot H$ are given by
$$c\Co_D c'\longmapsto c\ps{1}\ot c\ps{2}\Co_D c', \quad ~~~ c\ot h\longmapsto c\ps{1}\ot c\ps{2}\ot h.$$ Furthermore it is a right $H$-module map where the right $H$ module structures of $C\Co_D C$ and $C\ot H$ are given by
$$(c\Co_D c')\triangleleft h= c\Co_D c'\triangleleft h, \quad ~~~ (c\ot h)\triangleleft h'= c\ot hh'.$$ The first $H$-action introduced above is well-defined by \eqref{pi}. In fact the coextension $\pi: C\twoheadrightarrow D$ is Hopf Galois if $C\ot H$ and $C\Co_D C$ are isomorphic in  $^C\mathcal{M}_H$, the category of left $C$-comodules and right $H$-modules, by the canonical map $\b$. One notes that the categories $^C\mathcal{M}_H$  and $_A\mathcal{M}^H$  related to Hopf Galois extensions and coextensions are dual to each other.

\begin{lemma}\label{canonicalbijection} Let  $C(D)^{H}$ be a Hopf Galois coextension with canonical bijection $\beta$. Then  the following properties hold.
 \begin{enumerate}
\item[i)]  $\beta_{-}\ps{1}\ot \beta_{-}\ps{2}\triangleleft \beta_{+}= Id_{C\Co_D C}.$
 \item[ ii)] $\beta_{-}(c\ps{1}\Co_D c\ps{2}\triangleleft h) \ot \beta_{+}(c\ps{1}\Co_D c\ps{2}\triangleleft h)= c\ot h, \quad c\in C, h\in H.$
 \item[ iii)] $\beta_-(c_1\Co_D c_2)\triangleleft \beta_+(c_1\Co_D c_2)= \ve(c_1)c_2.$
 \item[iv)] $\ve(\beta_-(c_1\Co_D c_2))\ve(\beta_+(c_1\Co_D c_2))= \ve(c_1)\ve(c_2).$
 \item[v)] $\beta_-(c_1\Co_D c_2\triangleleft h)\ot \beta_+(c_1\Co_D c_2\triangleleft h)= \beta_-(c_1\Co_D c_2)\ot \beta_+(c_1\Co_D c_2)h$.
 \item[vi)] $[\beta_-(c_1\Co_D c_2)]\ps{1}\ot [\beta_-(c_1\Co_D c_2)]\ps{2}\ot \beta_+(c_1\Co_D c_2)=\\
 c_1\ps{1}\ot \beta_-(c_1\ps{2}\Co_D c_2)\ot \beta_+(c_1\ps{2}\Co_D c_2).$
 \item[vii)] $ \beta_-(c_1\Co_D c_2)\ps{1} \ot \beta_-(c_1\Co_D c_2)\ps{2}\ot \beta_-(c_1\Co_D c_2)\ps{3}\triangleleft \beta_+(c_1\Co_D c_2)\ps{1}\ot \beta_-(c_1\Co_D c_2)\ps{4}\triangleleft \beta_+(c_1\Co_D c_2)\ps{2}= c_1\ps{1}\ot c_1\ps{2}\Co_D c_2\ps{1}\ot c_2\ps{2}.$
 \end{enumerate}
\end{lemma}
\begin{proof}
  The relations i)  and ii) are equivalent to  $ \beta \beta^{-1}= \Id$ and  $\beta^{-1}\beta=\Id$, respectively. The relation iii) is obtained by applying $\ve\ot \Id$ on the both hand sides of i). The relation iv) is proved by applying $\ve\ot \ve$ on the both hand sides of i) and $H$-module coalgebra property of $C$. The relation v) is equivalent to the right $H$-module property of the map $\beta^{-1}$. The relation vi) is equivalent to the left $C$-comodule map property of the map $\beta^{-1}$ where the left $C$ comodule structures of $C\ot H$ and $C\Co_D C$ are given by $c\ot h\longmapsto c\ps{1}\ot c\ps{2}\ot h$ and $c_1\ot c_2\longmapsto c_1\ps{1}\ot c_1\ps{2}\ot c_2$, respectively. The relation vii) holds by applying $\Delta_C \ot \Delta_C$ on the both hand sides of relation (i) and using $H$-module coalgebra property of $C$.
\end{proof}

\subsection{ Hopf Galois coextensions as sources of stable anti-Yetter-Drinfeld modules}
In this subsection we show that Hopf Galois coextensions are  sources of  stable anti Yetter-Drinfeld modules.

\medskip

For any $D$-bicomodule $C$, we set
\begin{equation}
  C^D=\left\{c\in C, \quad  c\ns{0}\varphi(c\ns{1})= c\ns{0}\varphi(c\ns{-1}) \right \}_{\varphi\in D^*},
\end{equation}
where $D^*$ is the linear dual of $D$.
We set
\begin{equation}
  C_D=\frac{C}{W},
\end{equation}
where
$$W= \left\{c\ns{0}\varphi(c\ns{1})- c\ns{0}\varphi(c\ns{-1}), \quad c\in C \right \}_{\varphi\in D^*}.$$
For any  coalgebra coextension  $\pi: C\twoheadrightarrow D$, we precisely obtain
\begin{equation}\label{TS}
C^D:=\left\{c\in C, \quad  c\ps{1} \varphi(\pi(c\ps{2}))= c\ps{2}\varphi( \pi(c\ps{1})), \quad c\in C\right\}_{\varphi\in D^*},
\end{equation}
 and
\begin{equation}
W=\left\{  c\ps{1} \varphi(\pi(c\ps{2}))- c\ps{2}\varphi( \pi(c\ps{1})), \quad c\in C\right\}_{\varphi\in D^*}.
\end{equation}

\begin{lemma}\label{SAYDaction}
 Let $C(D)^H$ be a Hopf Galois coextension  with the corresponding $H$-action $\triangleleft:C\ot H\longrightarrow C$. Then $\triangleleft$ induces the following $H$-action, $\blacktriangleleft:C^D\ot H\longrightarrow C^D$.
 \end{lemma}
\begin{proof}
 It is enough to show that $c\triangleleft h\in C^D$ for all $c\ot h\in C^D\ot H$.  Indeed,
\begin{align*}
&(c\triangleleft h)\ps{1}\varphi(\pi(c\triangleleft h)\ps{2}))=c\ps{1}\triangleleft h\ps{1}\varphi(\pi(c\ps{2}\triangleleft h\ps{2}))\\
&=c\ps{1}\triangleleft h\ps{1}\varphi(\pi(c\ps{2}))\ve(h\ps{2})= c\ps{1}\triangleleft h\varphi(\pi(c\ps{2}))\\
&=c\ps{2}\triangleleft h\varphi(\pi(c\ps{1}))=c\ps{2}\triangleleft h\ps{2}\varphi(\pi(c\ps{1}))\ve(h\ps{1})\\
&=c\ps{2}\triangleleft h\ps{2}\varphi(\pi(c\ps{1}\triangleleft h\ps{1}))=(c\triangleleft h)\ps{2}\varphi(\pi(c\triangleleft h)\ps{1})).
\end{align*}
We use the $H$-module coalgebra property of $C$ in the first equality and $c\in C^D $ in the fourth equality.

\end{proof}

One can define a $D$-bicomodule structure on $C\Co_D C$ as follows;
\begin{align}\label{coalgebras}
  c\Co_D c'\longmapsto c\Co_D c'\ps{1}\ot \pi(c'\ps{2}),\quad c\Co_D c'\longmapsto \pi(c\ps{1})\ot c\ps{2}\Co_D c'.
\end{align}
It is easy to check that the  coactions defined in \eqref{coalgebras} are well-defined. We set
\begin{equation}
  (C\Box_D C)_D= \frac{C\Co_D C}{W},
\end{equation}
where
\begin{equation}\label{ww}
  W= \langle   c\ot c'\ps{1} \varphi(\pi(c'\ps{2}))- c\ps{2}\ot c'\varphi(\pi(c\ps{1}))  \rangle,
\end{equation}
 and  $\varphi\in D^*$, $c\ot c'\in C\Co_D C$. We denote the elements of the quotient by an over line. Although  $C\Co_D C$ is not a coalgebra, it is proved in \cite[section 6.4, page 93-95]{bal2} that the quotient space   $(C\Co_D C)_D$  is a coassociative coalgebra where the coproduct and counit maps are given by
\begin{equation}\label{coproduct-main}
  \Delta(\overline{c\ot c'})= \overline{c\ps{1}\Co_D c'\ps{2}}\ot \overline{c\ps{2}\Co_D  c'\ps{1}}. \qquad \ve(\overline{c\ot c'})=\ve(c)\ve(c').
\end{equation}

The following lemma can be similarly proved as \cite{hr2}[Lemma 4.4].
\begin{lemma}
 If $C(D)^H$ be a Hopf-Galois coextension, then $\beta$ induces a bijection $\bar{\beta}:C_D\ot H\longrightarrow(C\Box_D C)_D$ where
 $$\bar{\beta}(\bar{c}\ot h)=\overline{\beta(c\ot h)}.$$
 \end{lemma}

The following lemma can be similarly proved as \cite{hr2}[Lemma 4.5].
\begin{lemma}
Let  $C(D)^H$  be a Hopf-Galois coextension. Define
$$\kappa:=(\varepsilon\ot \Id_H)\circ \bar{\beta}^{-1}:(C\Box_DC)_D\longrightarrow H.$$ The map $\kappa$ is an anti coalgebra map.
\end{lemma}

The anti coalgebra map property of the map $\kappa$ is equivalent to
\begin{equation}
  \kappa(\overline{c\ot c'})\ps{1}\ot \kappa(\overline{c\ot c'})\ps{2}=\kappa(\overline{c\ps{2}\ot c'\ps{1}})\ot \kappa(\overline{c\ps{1}\ot c'\ps{2}}).
\end{equation}

The following lemma introduces some properties of the map $\kappa$.

\begin{lemma}\label{kappa}
If   $C(D)^H$  be a Hopf Galois coextension then the map $\kappa$ has the following properties.
\begin{enumerate}
\item[i)]$ \kappa(\ol{c\ps{1} \triangleleft h \Co_D c\ps{2} \triangleleft g})=\ve(c)S(h) g, \quad  c\in C,\; g,h\in H,$
\item[ii)] $ \kappa(\ol{c\Co_D c'})h=\kappa(\ol{c\Co_D c'\triangleleft h}), \quad h\in H,$
\item[iii)] $ \kappa(\ol{c\triangleleft h\Co_D c'})=S(h)\kappa(\ol{c\Co_D c'}), \quad h\in H,$
\item[iv)] $ c\ps{1}\triangleleft \kappa(\ol{c\ps{2}\Co_D c'})= \ve(c) c', \quad c\ot c'\in C\Co_D C$.
\end{enumerate}
\end{lemma}
\begin{proof}
The following computation proves  the relation i).
\begin{align*}
&\kappa(c\ps{1} \triangleleft h \Co_D c\ps{2} \triangleleft g)=\kappa(c\ps{1} \lt h\ps{1} \Co_D c\ps{2} \lt \ve(h\ps{2})g)\\
&=\kappa(c\ps{1} \lt h\ps{1} \Co_D c\ps{2} \lt h\ps{2}S(h\ps{3})g)=\kappa(c\ps{1} \lt h\ps{1} \Co_D (c\ps{2} \lt h\ps{2})\lt S(h\ps{3})g)\\
&=\kappa((c \lt h\ps{1})\ps{1} \Co_D (c \lt h\ps{1})\ps{2}\lt S(h\ps{2})g)=\ve(c \lt h\ps{1})S(h\ps{2})g\\
&=\ve(c)\ve(h\ps{1})S(h\ps{2})g=\ve(c)S(h)g.
\end{align*}
We used Lemma \ref{canonicalbijection}(ii) in the fifth equality. The relation ii) is obvious by the right $H$-module map property of $\beta^{-1}$ in Lemma \ref{canonicalbijection}(v). To prove iii), it is enough to show that the maps $m_H\circ (S_H\ot \kappa)$ and $\kappa\circ (\triangleleft\ot \Id_C)\circ(tw\ot \Id_C)$, appearing in both hand sides of iii),  have the same inverse in the convolution algebra ${\rm Hom_k(H\ot C\Box_D C,H)}$. Here $m_H$ denotes the multiplication map of $H$. To do this, first we show that the map ${m_H \circ tw\circ(\Id_H \ot  [S^{-1}\circ \kappa])}$ is a left inverse for ${ m_H\circ(S \ot\kappa)}$, with respect to the convolution product denoted by $\star$. Let ${h\ot \ol{ c \ot c' }\in H\ot (C \Box_D C)_D}$. We have;
\begin{align*}
&[m_H \circ tw\circ(\Id_H \ot  [S^{-1}\circ \kappa])\star m_H\circ(S \ot\kappa)](h\ot \ol{c \ot c'})\\
&=S^{-1}(\kappa(\ol{c\ps{1}\ot c'\ps{2}})) h\ps{1} S(h\ps{2})\kappa(\ol{c\ps{2}\ot c'\ps{1}})\\
&=\ve(h) S^{-1}(\kappa(\ol{c\ps{1}\ot c'\ps{2}})) \kappa(\ol{c\ps{2}\ot c'\ps{1}})\\
&=\ve(h)S^{-1}[\kappa((\ol{c\ot c'})\ps{1})]\kappa((\ol{c\ot c'})\ps{2})\\
&=\ve(h)S^{-1}[\kappa(\ol{c\ot c'})\ps{2}]\kappa(\ol{c\ot c'})\ps{1}\\
&=\ve(h)\ve(\kappa(\ol{c\ot c'}))\\
&=\ve(h)\ve(\kappa(\ol{c\ps{1}\ot c'\ps{2}}))\ve(\kappa(\ol{c\ps{2}\ot c'\ps{1}}))\\
&= \ve(h)\ve(c)\ve(c')1_H=\eta \circ \ve (h\ot \ol{c \ot c'}).
\end{align*}
Here $\eta$  is the unit map of $H$. In the previous calculation, we use the anti coalgebra map property of $\kappa$ in the fourth equality and \ref{canonicalbijection}(iv) in the seventh equality. Now we show that the map ${m_H \circ tw\circ(\Id_H \ot  [S^{-1}\circ \kappa])}$ is a right inverse for the map ${ m_H\circ(S \ot\kappa)}$. For the convenience in the following computation, we denote $\beta^{-1}(c\ot d)=\beta_- \ot \beta_+$.

\begin{align*}
  &[ m_H\circ(S \ot\kappa) \star m_H \circ tw\circ(\Id_H \ot  [S^{-1}\circ \kappa])](h\ot \ol{c \ot c'})\\
  & =S(h\ps{1})\kappa(\ol{c\ps{1}\ot c'\ps{2}}) S^{-1}(\kappa(\ol{c\ps{2}\ot c'\ps{1}}))h\ps{2}\\
  &=S(h\ps{1})\kappa(\ol{\beta_-\ps{1}\ot \beta_-\ps{4}\triangleleft \beta_+\ps{2}}) S^{-1}(\kappa(\ol{\beta_-\ps{2}\ot \beta_-\ps{3}\triangleleft \beta_+\ps{1}}))h\ps{2}\\
  &=S(h\ps{1})\kappa(\ol{\beta_-\ps{1}\ot \beta_-\ps{3}\triangleleft \beta_+\ps{2}}) S^{-1}(\ol{\ve(\beta_-\ps{2})\beta_+\ps{1}})h\ps{2}\\
  &=S(h\ps{1})\kappa(\ol{\beta_-\ps{1}\ot \beta_-\ps{2}\triangleleft \beta_+\ps{2}}) S^{-1}(\ol{\beta_+\ps{1}})h\ps{2}\\
 & =S(h\ps{1})\ve(\beta_-)\beta_+\ps{2} S^{-1}(\ol{\beta_+\ps{1}})h\ps{2}=\ve(\beta_-)\ve(\beta_+) S(h\ps{1})h\ps{2}\\
  &=\ve(c)\ve(c')\ve(h)1_H= \eta\circ \ve(h\ot \ol{c\ot c'}).\\
\end{align*}
We used Lemma \ref{canonicalbijection}(vii) in the second equality, Lemma \ref{canonicalbijection}(ii) and definition of $\kappa$ in the third  and fifth equalities and Lemma \ref{canonicalbijection}(iv) in the penultimate equality.
Therefore we have shown that ${m_H \circ tw\circ(\Id_H \ot  [S^{-1}\circ \kappa])}$ is a two-sided inverse for ${ m_H\circ(S \ot\kappa)}$.
Now we check  ${m_H \circ tw\circ(\Id_H \ot  [S^{-1}\circ \kappa])}$ is  a left inverse for the map ${ \kappa\circ(\triangleleft\ot \Id_C)\circ(tw\ot \Id_C)}$ with respect to the convolution product.
\begin{align*}
&[m_H \circ tw\circ(H \ot  [S^{-1}\circ \kappa])\star \kappa\circ(\triangleleft\ot \Id_C)\circ(tw\ot \Id_C)](h\ot \ol{c \ot c'})\\
&=m_H \circ tw\circ(h\ps{1} \ot  S^{-1}(\kappa(c\ps{1} \ot c'\ps{2})))\kappa\circ(\triangleleft\ot C)(c\ps{2}\ot h\ps{2}\ot c'\ps{1})\\
&=S^{-1}(\kappa(c\ps{1} \ot c'\ps{2}))h\ps{1} \kappa(c\ps{2}\lt h\ps{2}\ot c'\ps{1})\\
&=S^{-1}(\kappa(\ol{\beta_-\ps{1} \ot \beta_-\ps{4}\lt \beta_+\ps{2}}))h\ps{1}\kappa(\ol{\beta_-\ps{2}\lt h\ps{2} \ot \beta_-\ps{3}\lt \beta_+\ps{1}})\\
&=S^{-1}(\kappa(\ol{\beta_-\ps{1} \ot \beta_-\ps{3}\lt \beta_+\ps{2}}))h\ps{1}\ve(\beta_-\ps{2})S(h\ps{2})\beta_+\ps{1}\\
&=S^{-1}(\kappa(\ol{\beta_-\ps{1} \ot \beta_-\ps{2}\lt \beta_+\ps{2}}))\ve(h)\beta_+\ps{1}\\
&=S^{-1}(\ve(\beta_-)\beta_+\ps{2})\ve(h)\beta_+\ps{1}\\
&=\ve(\beta_-)\ve(h)\ve(\beta_+)1_H=\ve(c)\ve(h)\ve(c')1_H=\iota \circ \ve (h\ot \ol{c \ot c'}).
\end{align*}
We used  Lemma \ref{canonicalbijection}(vii) in the third equality,  Lemma \ref{kappa}(i) in the fourth equality, definition of $\kappa$ and Lemma \ref{canonicalbijection}(ii) in the sixth equality and Lemma \ref{canonicalbijection}(iv) in the eighth equality.
Here we show that $m_H \circ tw\circ(\Id_H \ot  [S^{-1}\circ \kappa])$ is a right inverse for $\kappa\circ(\triangleleft\ot \Id_C)\circ(tw\ot \Id_C)$.

\begin{align*}
&[\kappa\circ(\triangleleft\ot \Id_C)\circ(tw\ot \Id_C)\star m_H \circ tw\circ(\Id_H \ot  [S^{-1}\circ \kappa])](h\ot \ol{c \ot c'})\\
&=\kappa(c\ps{1}\triangleleft h\ps{1} \ot c'\ps{2})S^{-1}(\kappa(c\ps{2}\ot c'\ps{1}))h\ps{2}\\
&=\kappa(\ol{\beta_-\ps{1}\lt h\ps{1}  \ot \beta_-\ps{4}\lt \beta_+\ps{2}})S^{-1}(\kappa(\ol{\beta_-\ps{2} \ot \beta_-\ps{3}\lt \beta_+\ps{1}}))h\ps{2}\\
&=\kappa(\ol{\beta_-\ps{1}\lt h\ps{1}  \ot \beta_-\ps{3}\lt \beta_+\ps{2}})S^{-1}(\ve(\beta_-\ps{2})\beta_+\ps{1})h\ps{2}\\
&=\kappa(\ol{\beta_-\ps{1}\triangleleft h\ps{1}  \ot \beta_-\ps{2}\lt \beta_+\ps{2}}))S^{-1}(\beta_+\ps{1})h\ps{2}\\
&=\ve(\beta_-)S(h\ps{1})\beta_+\ps{2}S^{-1}(\beta_+\ps{1})h\ps{2}\\
&=\ve(\beta_-)\ve(\beta_+)S(h\ps{1})h\ps{2}\\
&=\ve(c)\ve(c')\ve(h)1_H=\iota \circ \ve (h\ot \ol{c \ot c'}).
\end{align*}
 We used Lemma \ref{canonicalbijection}(vii) in the second equality,  Lemma \ref{kappa}(i) and the definition of $\kappa$ in the third equality, Lemma \ref{kappa}(i) in the fifth equality and Lemma \ref{canonicalbijection}(iv) in the penultimate equality.
 Therefore ${ m_H \circ tw\circ(\Id_H \ot  [S^{-1}\circ \kappa])}$ is a two-sided inverse for ${ \kappa\circ(\triangleleft\ot \Id_C)\circ(tw\ot \Id_C)}$ with respect to the convolution product. Therefore we have shown
$$\kappa\circ(\mu_C\ot C)\circ(tw\ot C)=\mu_H\circ(S \ot\kappa).$$ To prove the relation iv), we apply $ \Id\ot \ve \ot \Id$ on the both hand sides of  Lemma \ref{canonicalbijection}(vi) and therefore we obtain;
\begin{align}\label{v-v-v}\nonumber
& \b_-(\overline{c\ot c'})\ot \b_+(\overline{c\ot c'})= \ve(\b_-(\overline{c\ps{2}\ot c'}))c\ps{1}\ot \b_+(\overline{c\ps{2}\ot c'})\\
 &= c\ps{1}\ot \kappa(\overline{c\ps{2}\ot c'}).
\end{align}
By applying the right action of $H$  on the previous equation we obtain;
\begin{align*}
 & c\ps{1}\triangleleft \kappa(\overline{c\ps{2}\ot c'})= \ve(\b_-(\overline{c\ps{2}\ot c'}))c\ps{1}\triangleleft \b_+(\overline{c\ps{2}\ot c'})\\
 &= \b_-(\overline{c\ot c'})\triangleleft \b_+(\overline{c\ot c'})=\ve(c)c'.
\end{align*}
We used the Lemma \ref{canonicalbijection}(iii) on the last equality.

\end{proof}

\begin{lemma}\label{lemma C^D}
If $C(D)^H$  be a Hopf Galois coextension, then  $C^D$ is a right $(C\Box_D C)_D$-comodule and a left $H$-comodule by the following coactions,
\begin{align*}
  \blacktriangledown^{C^D}:C^D\longrightarrow C^D\ot (C\Box_DC)_D, \quad c\longmapsto c\ps{2} \ot \ol{c\ps{3}\Co_D c\ps{1}},
\end{align*}
 and
 \begin{align}\label{SAYDcoaction}
 \triangledown^{C^D}:C^D\longrightarrow H\ot C^D, \quad c\longmapsto \kappa(\ol{c\ps{3}\Co_D c\ps{1}})\ot c\ps{2}.
 \end{align}
 \end{lemma}
\begin{proof}
First we show that the coaction $\blacktriangledown^{C^D}$ is well-defined.
The following computation proves $c\ps{2}\ot \overline{c\ps{3}\ot c\ps{1}}\in C^D\ot C\ot C$.
  \begin{align*}
    &c\ps{3}\varphi(\pi(c\ps{2}))\ot \overline{c\ps{4}\ot c\ps{1}}=c\ps{3}\ot \overline{c\ps{4}\ot c\ps{1}\varphi(\pi(C\ps{2}))}\\
    &=c\ps{3}\ot \overline{c\ps{4}\ps{2}\varphi(\pi(c\ps{4}\ps{1})\ot \ve(c\ps{1})c\ps{2}}=c\ps{2}\varphi(\pi(c\ps{3}))\ot \overline{c\ps{4}\ot c\ps{1}}.
  \end{align*}
  We used the definition of $W$  given in \eqref{ww}.  Furthermore since $C$ is a coalgebra on a field then $c\in C^D$ implies   $c\ps{1}\ot \pi(c\ps{2})=  c\ps{2}\ot \pi(c\ps{1})$. Therefore we have;
  \begin{equation}
    c\ps{1}\ot c\ps{2}\ot c\ps{3}\ot \pi(c\ps{4})= c\ps{2}\ot c\ps{3}\ot c\ps{4}\ot \pi(c\ps{1}),
  \end{equation}
  which implies
  $$c\ps{2}\ot \overline{c\ps{3}\ot c\ps{1}}\in C\ot C\Co_D C.$$
The following computation shows  that $\blacktriangledown^{C^D}$ defines a coassociative coaction.
\begin{align*}
&[\blacktriangledown^{C^D}\ot \Id_{(C\Box_DC)_D}]\circ \blacktriangledown^{C^D}(c)=(\blacktriangledown^{C^D}\ot \Id_{(C\Box_DC)_D})(c\ps{2} \ot \ol{c\ps{3}\ot c\ps{1}})\\
&=(c\ps{3}\ot\ol{c\ps{4}\ot c\ps{2}} \ot \ol{c\ps{5}\ot c\ps{1}})=c\ps{2} \ot \ol{c\ps{3}\ps{1}\ot c\ps{1}\ps{2}}\ot \ol{c\ps{3}\ps{2}\ot c\ps{1}\ps{1}}\\
&=(\Id_{C^D}\ot\D)(c\ps{2} \ot \ol{c\ps{3}\ot c\ps{1}})=(\Id_{C^D}\ot\D)\circ \blacktriangledown^{C^D}(c).
\end{align*}
The counitality of the coaction induced by $\blacktriangledown^{C^D}$ can be shown  as follows.
$$(\Id_{C^D}\ot \ve)\circ\varrho^{C^D}(c)=(\Id_{C^D}\ot \ve)(c\ps{2} \ot \ol{c\ps{3}\ot c\ps{1}})=c\ps{2}\ve(c\ps{3})\ve(c\ps{1})=\Id_{C^D}(c).$$
Therefore $C^D$ is a right coassociative and counital $ (C\Box_DC)_D$-comodule.\\
The following computation shows that $\triangledown^{C^D}$ defines a coassociative $H$-coaction on $C^D$.
\begin{align*}
&(\Id_H\ot \triangledown^{C^D})\circ \triangledown^{C^D}(c)=(\Id_H\ot \triangledown^{C^D})(\kappa(\ol{c\ps{3}\ot c\ps{1}})\ot c\ps{2})\\
&=\kappa(\ol{c\ps{5}\ot c\ps{1}})\ot \kappa(\ol{c\ps{4}\ot c\ps{2}}) \ot c\ps{3}\\
&=\kappa(\ol{c\ps{3}\ot c\ps{1}})\ps{1}\ot \kappa(\ol{c\ps{3}\ot c\ps{1}})\ps{2} \ot c\ps{2}\\
&=(\D\ot \Id_H)(\kappa(\ol{c\ps{3}\ot c\ps{1}})\ot c\ps{2})=(\D\ot \Id_H)\circ \triangledown^{C^D}(c).
\end{align*}
We used the  anti coalgebra map property of $\kappa$ in the third equality. To prove the counitality of the coaction for all $c \in C^D$  we have;
\begin{align*}
&(\ve\ot \Id_{C^D})\circ\varrho^{C^D}(c)=(\ve\ot \Id_{C^D})(\kappa(\ol{c\ps{3}\ot c\ps{1}})\ot c\ps{2})=\\
&\ve\circ\kappa(\ol{c\ps{3}\ot c\ps{1}})\ot c\ps{2}=\ve(c\ps{3})\ve(c\ps{1})\ot c\ps{2}=id_{C^D}(c).
\end{align*}
We used the  anti coalgebra map property of $\kappa$ in the third equality. Therefore $C^D$ is a left $H$-comodule.
\end{proof}

\begin{lemma}
  Let $C(D)^H$  be Hopf Galois coextension. Then  $C_D$ is a left $(C\Box_D C)_D$-comodule and a right $H$-comodule by the following coactions,
\begin{align*}
  \blacktriangledown^{C_D}:C_D\longrightarrow  (C\Box_DC)_D\ot C_D, \quad c\longmapsto \ol{c\ps{1}\Co_D c\ps{3}}\ot c\ps{2},
\end{align*}
 and
 \begin{align}\label{SAYDcoaction}
 \triangledown^{C_D}:C_D\longrightarrow  C_D\ot H, \quad c\longmapsto  c\ps{2}\ot  \kappa(\ol{c\ps{1}\Co_D c\ps{3}}).
 \end{align}
\end{lemma}
\begin{proof}
   Since $C_D= \frac{C}{W}$ is a coalgebra  this coaction is well-defined. The following computation shows that this coaction is coassociative.
  \begin{align*}
  &(\Delta_{{C\Co_D C}_D}\ot \Id_C)\circ \blacktriangledown^{C_D}(c)\\
    &=(\overline{t\ps{1}\ot t\ps{3}})\ps{1}\ot (\overline{t\ps{1}\ot t\ps{3}})\ps{2}\ot t\ps{2}=t\ps{1}\ps{1}\ot t\ps{3}\ps{2}\ot t\ps{1}\ps{2}\ot t\ps{3}\ps{1}\ot t\ps{2}\\
    &=t\ps{1}\ot t\ps{5}\ot t\ps{2}\ot t\ps{4}\ot t\ps{3}= t\ps{1}\ot t\ps{3}\ot t\ps{2}\ps{1}\ot t\ps{2}\ps{3}\ot t\ps{2}\ps{2}\\
    &=(\Id_{{C\Co_D C}_{D}}\ot \blacktriangledown_{C})\circ \blacktriangledown_{C}.
  \end{align*}
  The counitality of the coaction is obvious. Similar to Lemma \ref{lemma C^D} we can show that the anti coalgebra map $\kappa$ turns this coaction to a right coaction of $H$ on $C_D$.
\end{proof}
 The following theorem shows that the Hopf Galois coextensions of coalgebras are the sources of stable anti Yetter-Drinfeld modules.

\begin{theorem}\label{main01}
Let ${ C(D)^H}$ be a Hopf Galois coextension. Then $ C^D$ is a right-left SAYD module over $H$ by the coaction $\triangledown^{C^D}$ defined in  \eqref{SAYDcoaction} and the action $\blacktriangleleft$ defined in Lemma \ref{SAYDaction}.
\end{theorem}
\begin{proof}
The   AYD condition  holds because for all $c\in C^D$ and $ h\in H$, we have
\begin{align*}
  &S(h\ps{3})c\ns{-1}h\ps{1}\ot c\ns{0}\blacktriangleleft h\ps{2}=S(h\ps{3})\kappa(\ol{c\ps{3}\ot c\ps{1}})h\ps{1}\ot c\ps{2}\blacktriangleleft h\ps{2}\\
  &=\kappa(\ol{c\ps{3}\blacktriangleleft h\ps{3}\ot c\ps{1}})h\ps{1}\ot c\ps{2}\blacktriangleleft h\ps{2}=\kappa(\ol{c\ps{3}\blacktriangleleft h\ps{3}\ot c\ps{1}\blacktriangleleft h\ps{1}})\ot c\ps{2}\blacktriangleleft h\ps{2} \\
  &=\kappa(\ol{(c\blacktriangleleft h)\ps{3}\ot (c\blacktriangleleft h)\ps{1}})\ot (c\blacktriangleleft h)\ps{2}= \triangledown^{C^D}(c\blacktriangleleft h).
\end{align*}

We used  Lemma \ref{kappa}(iii) in the second equality, Lemma \ref{kappa}(ii) in the third equality and  right $H$-module coalgebra property of $C$ in the fourth equality. The following computation shows the stability condition.
\begin{align*}
 &c\ns{0}\blacktriangleleft c\ns{-1} = c\ps{2} \blacktriangleleft \kappa(\ol{c\ps{3}\ot c\ps{1}})=\\
 &c\ps{2}\ps{1}\blacktriangleleft \kappa(\ol{c\ps{2}\ps{2}\ot c\ps{1}})= \ve(c\ps{2})c\ps{1}= c.
\end{align*}
The penultimate equality holds by Lemma \ref{kappa}(iv).
\end{proof}
The following statement which is proved in \cite{hr2}[Remark 5.4] will be used in the sequel subsection.
\begin{lemma}
For any coalgebra coextension $\pi: C\longrightarrow D$ we have;
  $$C^D\cong D\Co_{D^e} C,$$ where the isomorphism is given by
  $$\xi( d\Co_{D^e} c)= \varepsilon_D(d)c, \quad \text{and}  \quad \xi^{-1}(c)= \pi(c\ps{1})\Co_{D^e} c\ps{2}.$$
\end{lemma}

\subsection{Cohomology theories related to  Hopf Galois coextensions}
In this subsection we show that the Hopf cyclic cohomology of the Hopf algebra involved in  the coextensions with coefficients in the SAYD module $C^D$  is isomorphic to the relative cyclic cohomology of coalgebra coextensions.

\medskip
\begin{definition}
Let $D$ be a coalgebra and let $M$ and $N$ be a $D$-bicomodules. The cyclic cotensor product of $M$ and $N$ is defined by
\begin{equation}
 M\wh{\Co}_D N :=(M \Co_D N) \Co_{D^e} D\cong D\Co_{D^e} (M \Co_D N).
\end{equation}
  The cyclic cotensor product also can be defined for a finite number of arbitrary $D$-bicomodules $M_1,\ldots,M_n$ as follows;
\begin{equation}
(M_1\wh{\Co}_d\ldots\wh{\Co}_D M_n):=(M_1\Co_D\ldots\Co_D M_n)\Co_{D^e}D.
\end{equation}
\end{definition}
One has
\begin{equation}
(M_1\wh{\Co}_d\ldots\wh{\Co}_D M_n)\cong (M_1{\Co}_d\ldots{\Co}_D M_n)^D.
\end{equation}
Let $\pi:C\twoheadrightarrow D$ be a coalgebra coextension and $M$ be a $C$-bicomodule. Therefore $M$ is a $D$-bicomodule by the following comodule structures;
\begin{equation}
  c\longmapsto \pi(c\ns{-1})\ot c\ns{0}, \quad c\longmapsto c\ns{0}\ot \pi(c\ns{1}).
\end{equation}

  For all $n\in \mathbb{N}$ and  $0\leq i\leq n$,
we define cofaces
\begin{equation}
d_i:M\wh{\Co}_DC^{\wh{\Co}_D n}\longrightarrow M\wh{\Co}_DC^{\wh{\Co}_D (n+1)},
\end{equation}
 which are given by
\begin{equation}\label{coface}
d_i(m\Box_D c_1 \Box_D \ldots \Box_D c_n )=
\begin{cases}
m\ns{0}\Box_D m\ns{1} \Box_D c_1 \Box_D \ldots  \Box_D c_n, & \text{$i=0$},\\
m\Box_D c_1 \Box_D \ldots \Box_D \D(c_i) \Box_D \ldots \Box_D  c_n ,& \text{$0 < i < n$},\\
m\ns{0} \Box_D c_1 \Box_D \ldots  \Box_D c_n \Box_D m\ns{-1}, & \text{$i=n$}.
\end{cases}
\end{equation}
Also for $0\leq i \leq n-1$ we define the codegeneracies
\begin{equation}\label{codeg}
s_i:M\wh{\Co}_DC^{\wh{\Co}_D n}\longrightarrow M\wh{\Co}_DC^{\wh{\Co}_D (n-1)},
\end{equation}
 which are given by
\begin{equation}
s_i(m\Box_D c_1 \Box_D \ldots \Box_D c_n ):= m\Box_D c_1 \Box_D \ldots \Box_D \ve(c_{i+1}) \Box_D \ldots  \Box_D c_n.
\end{equation}
One can easily check that for $0\leq i \leq n-1$, the cofaces $d_i$ are  well-defined. To show that the last coface is well-defined, one notes that the condition

$$m\ot c_1\ot \cdots \ot c_n\in (M\Co C{\Co}\cdots {\Co}C)^D,$$ implies
\begin{equation}\label{symmetry}
  \pi(m\ns{-1})\ot m\ns{0}\ot c_1\ot \cdots \ot c_n= \pi(c_n\ps{2})\ot m\ot c_1\ot \cdots c_{n-1}\ot c_n\ps{1}.
\end{equation}
The following computation proves that the last coface is well-defined.
\begin{align*}
  &m\ns{0}\ot c_1\ot \cdots \ot c_{n-1}\ot c_n\ps{1}\ot \pi(c_n\ps{2})\ot m\ns{-1} \\
  &=m\ns{0}\ns{0}\ot c_1\ot \cdots c_{n-1} \ot c_n\ot \pi(m\ns{-1})\ot m\ns{0}\ns{-1}\\
  &=m\ns{0}\ot c_1\ot \cdots c_{n-1}\ot\pi(m\ns{-1}\ps{1})\ot m\ns{-1}\ps{2}.
\end{align*}
We used \eqref{symmetry} in the first equality and the coassociativity of $C$-coaction of $M$ in the second equality.
\begin{theorem}\label{dual-thm}
Let $\pi:C\twoheadrightarrow D$ be a coalgebra coextension and $M$ be a $D$-bicomodule.

a) The module $C^n(C(D),M)=M\wh{\Co}_DC^{\wh{\Co}_Dn}$ is a cosimplicial object with the cofaces and codegeneracies  as defined in \eqref{coface} and \eqref{codeg}.

b)
Let $M=C$ and set
$$t:C\wh{\Co}_DC^{\wh{\Co}_Dn}\longrightarrow C\wh{\Co}_DC^{\wh{\Co}_Dn},$$ which is given by
\begin{equation}
 t_n( c_0\ot c_1\ot \ldots \ot c_n)=  c_1 \ot \ldots \ot c_n \ot c_0.
 \end{equation}
  Then $\left(C^n(C(D),C), d_i,s_i,t_n\right)$ is a cocyclic module.
\end{theorem}
The cyclic cohomology of the preceding cocyclic module is called relative cyclic cohomology of coalgebra coextension $C(D)$. The Theorem \ref{dual-thm} is the dual statement of \cite{JS}[Theorem 1.5]. In fact

\begin{equation}
  C^n(C(D),C)=       D\Co_{D^e} \underbrace{  C\Co_D\cdots \Co_D C}_{n+1 ~times}.
\end{equation}

For any Hopf Galois coextension  $C(D)^H$  the canonical map $\beta:C\ot H \longrightarrow C\Box_D C$ induces the following  bijection,
$$\beta^n: C\ot H^{\ot n}\longrightarrow  C^{\Box_D (n+1)}$$
$$c\ot h_1 \ot \dots \ot h_n \longmapsto c\ps{1} \ot c\ps{2} \triangleleft h_1\ps{1} \ot c\ps{3} \triangleleft h_1\ps{2} h_2\ps{1} \ot \dots \ot c\ps{n+1} \triangleleft h_1\ps{n}\dots h_{n-1}\ps{2}h_n.$$ By applying   $D\Co_{D^e}-$ on the both sides of the map $\b^{ n}$    we obtain the following  isomorphism of $C$-bicomodules;
\begin{equation*}
\vp :  C^D\ot H^{\ot n} \cong D\Co_{D^e} C\ot H^{\ot n} \longrightarrow D \Box_{D^e}C\Co_D C^{\Box_D n}\cong  C^{\Box_D n+1},
\end{equation*}
which is given by
\begin{align}\label{isom}
&c \ot h_1 \ot  \dots \ot  h_n \longmapsto  \nonumber \\
 &c\ps{1}\ot c\ps{2} \triangleleft h_1\ps{1} \ot c\ps{3}\triangleleft h_1\ps{2}h_2\ps{1} \ot \dots \ot c\ps{n}\triangleleft h_1\ps{n-1}\dots h_{n-1}\ps{1} \ot \nonumber\\
 &  ~~~~~~~~~~~~~~~~~~~ c\ps{n+1}\triangleleft h_1\ps{n}\dots h_{n-1}\ps{2} h_{n}.
\end{align}
The inverse map
$$\vp^{-1}:  C^{\Box_D n+1} \longrightarrow  C^D\ot H^{\ot n},$$
is given by
\begin{align}
&c_0 \ot \dots \ot c_n \longmapsto \nonumber \\
&c_0\ps{1} \ot \kappa(c_0\ps{2}\ot c_1\ps{1}) \ot \kappa(c_1\ps{2}\ot c_2\ps{1}) \ot \dots \ot \kappa(c_{n-2}\ps{2} \ot c_{n-1}\ps{1}) \ot \nonumber\\
&\kappa(c_{n-1}\ps{2}\ot c_n).
\end{align}

   By Theorem \ref{main01}, the subspace  $C^D$ is a right-left SAYD module over $H$. Using  the cocyclic module \eqref{dual-Hopf-cocyclic-module}, the Hopf cyclic cohomology of $H$ with coefficients in $C^D$ is computed by the following cocyclic module;
\begin{align}\nonumber
  &\delta_i(c\ot \widetilde{h})=  h_1\ot \cdots h_i\ot 1_H\ot h_{i+1}\ot \cdots h_n\ot c, \quad 0\leq i\leq n,\nonumber \\
  &\delta_{n+1}(c\ot \widetilde{h})=h_1\ps{1}\ot\cdots \ot h_n\ps{1}\ot S(h_1\ps{2}\cdots h_n\ps{2})\kappa(\overline{c\ps{3}\Co_D c\ps{1}}) \ot c\ps{2},\nonumber \\
  &\sigma_i(c\ot \widetilde{h})= h_1\ot \cdots\ot h_i h_{i+1}\ot \cdots \ot h_n\ot c,\nonumber \\
  &\sigma_n(c\ot \widetilde{h})=  h_1\ot \cdots \ot h_{n-1}\ve(h_n)\ot c,\nonumber \\
  &\tau_n(c\ot \widetilde{h})= h_2\ps{1}\ot\cdots\ot h_n\ps{1}\ot S(h_1\ps{2} \cdots h_n\ps{2})\kappa(\overline{c\ps{3}\Co_D c\ps{1}})\ot c\ps{2}\triangleleft h_1\ps{1},
\end{align}
where  $c\ot \widetilde{h}=c\ot h_1\ot \cdots \ot h_n$.   We denote the above cocyclic module by $C^n(H, C^D)$.
Now using the isomorphism $\vp$ given in \eqref{isom} we define the following map,
\begin{equation}
 \psi: C^n(H,C^D)\longrightarrow C^n(C(D),C),
\end{equation}
which is given by
\begin{align}\label{isom2}
& h_1  \dots \ot  h_n\ot c \longmapsto  \nonumber \\
 &c\ps{1}\ot c\ps{2} \triangleleft h_1\ps{1} \ot c\ps{3}\triangleleft h_1\ps{2}h_2\ps{1} \ot \dots \ot c\ps{n}\triangleleft h_1\ps{n-1}\dots h_{n-1}\ps{1} \ot \nonumber\\
 & c\ps{n+1}\triangleleft h_1\ps{n}\dots h_{n-1}\ps{2} h_{n},
\end{align}
with the inverse map;
\begin{align}
&c_0 \ot \dots \ot c_n \longmapsto \nonumber \\
& \kappa(c_0\ps{2}\ot c_1\ps{1}) \ot \kappa(c_1\ps{2}\ot c_2\ps{1}) \ot \dots \ot \kappa(c_{n-2}\ps{2} \ot c_{n-1}\ps{1}) \ot \nonumber\\
&\kappa(c_{n-1}\ps{2}\ot c_n)\ot c_0\ps{1}.
\end{align}
Now we are ready to state the main result.
\begin{theorem}
Let $C(D)^H$ be a Hopf Galois coextension.  The map $\psi$ introduced in \eqref{isom2} induces an isomorphism between the cocyclic modules $C^n(H,C^D)$ and $C^n(C(D),C)$ and we obtain
\begin{equation}
  HC^n(H,C^D)\cong HC^n(C(D),C).
\end{equation}
\end{theorem}
\begin{proof}
It is straight forward and easy  to prove that the map $\psi$ commutes with all cofaces, except the last one, and also codegeneracies. Here we show that $\psi$ commutes with the cyclic operator and therefore with the last coface map.
\begin{align*}
  & t_n(\psi(h_1 \ot \dots \ot  h_n\ot c ))  \nonumber \\
 &=t_n(c\ps{1}\ot c\ps{2} \triangleleft h_1\ps{1} \ot c\ps{3}\triangleleft h_1\ps{2}h_2\ps{1} \ot \dots \ot c\ps{n}\triangleleft h_1\ps{n-1}\dots h_{n-1}\ps{1} \ot\\
 &~~~~~~~~~~~~~~~~~~~~~~~~~~~~~~~~~~~~~~~~~ c\ps{n+1}\triangleleft h_1\ps{n}\dots h_{n-1}\ps{2} h_{n})\\
 &= c\ps{2} \triangleleft h_1\ps{1} \ot c\ps{3}\triangleleft h_1\ps{2}h_2\ps{1} \ot \dots \ot c\ps{n}\triangleleft h_1\ps{n-1}\dots h_{n-1}\ps{1} \ot\\
 &~~~~~~~~~~~~~~~~~~~~~~~~~~~~~~~~~~~~~~~~~ c\ps{n+1}\triangleleft h_1\ps{n}\dots h_{n-1}\ps{2} h_{n}\ot c\ps{1}\\
 &=c\ps{2}h_1\ps{1}\ot c\ps{3}h_1\ps{2}h_2\ps{1}\ot \cdots \ot c\ps{n+1}h_1\ps{n}\cdots h_{n-1}\ps{1}h_n\ot \\
  &~~~~~~~~~~~~~~~~~c\ps{n+2}\triangleleft \kappa(\overline{c\ps{n+3}\Co_D c\ps{1}})\\
 &=c\ps{2}h_1\ps{1}\ot c\ps{3}h_1\ps{2}h_2\ps{1}\ot \cdots \ot c\ps{n+1}h_1\ps{n}\cdots h_n\ps{1} \ot\\
  &~~~~~~~~~~~~~~~c\ps{n+2}h_1\ps{n+1}h_2\ps{n}\cdots h_n\ps{2}S(h_n\ps{3})\cdots  S(h_1\ps{n+2})\kappa(\overline{c\ps{n+3}\Co c\ps{1}})\\
 &=c\ps{2}h_1\ps{1}\ot c\ps{3}h_1\ps{2}h_2\ps{1}\ot \cdots \ot c\ps{n+1}h_1\ps{n}\cdots h_n\ps{1}\ot \\
 &~~~~~~~~~~~~~~~~~~~~~c\ps{n+2}h_1\ps{n+1}h_2\ps{n}\cdots h_n\ps{2}S(h_1\ps{n+2}\cdots h_n\ps{3})\kappa(\overline{c\ps{n+3}\Co_D c\ps{1}})\\
 &= (c\ps{2}h_1\ps{1})\ps{1}\ot (c\ps{2}h_1\ps{1})\ps{2}h_2\ps{1}\ot \cdots \ot (c\ps{2}h_1\ps{1})\ps{n}h_2\ps{1}\ps{n-1}\cdots h_n\ps{1}\ps{1} \ot\\
 &~~~~~~~~~~~~~~~~~~~~ (c\ps{2}h_1\ps{1})\ps{n+1}h_2\ps{1}\ps{n}\cdots h_n\ps{1}\ps{2}S(h_1\ps{2}\cdots h_n\ps{2})\kappa(\overline{c\ps{3}\Co_D c\ps{1}})  \\
 &=\psi(h_2\ps{1}\ot\cdots\ot h_n\ps{1}\ot S(h_1\ps{2} \cdots h_n\ps{2})\kappa(\overline{c\ps{3}\Co_D c\ps{1}})\ot c\ps{2}\triangleleft h_1\ps{1})\\
 &=\psi\tau(h_1 \ot \dots \ot  h_n\ot c ).\\
\end{align*}
We used the Lemma \ref{kappa}(iv) in third equality.
\end{proof}
We summarize  the headlines of this section in the following remark which are the dual ones in Remark \ref{remark}.

\begin{remark}{\rm
For any Hopf Galois coextension $C(D)^H$, the following statements hold.
\begin{enumerate}
 \item [i)] The relative cocyclic module of coalgebra coextensions is a subspace which is given by,
 $$D\ot_{D^e}\underbrace{ C\Co_D \cdots \Co_D C}_{n+1~times}= C\widehat{\Co}_D\cdots \widehat{\Co}_D C= (C\Co_D\cdots \Co_D C)^D.$$
 \item[ii)] The objects $C\ot H$ and $C\Co_D C$ are isomorphic in the category of $^C\mathcal{M}_H$.
 \item [iii)] The quotient space  $C_D=\frac{C}{W}$ is a coalgebra.
 \item[iv)] The quotient space $(C\Co_D C)_D$ is a coalgebra.
 \item [v)] The quotient space  $C_D$ is a left $(C\Co_D C)_D$-comodule and a right $H$-comodule.
  \item [vi)] The subspace  $C^D$ is a right $(C\Co_D C)_D$-comodule and a left $H$-comodule.
 \item [vii)]The subspace $C^D\cong D\Co_{D^e} C$ is a right-left SAYD module over $H$.
 \item [viii)] $HC^*(H,C_D)\cong HC^*(C(D),C)$.
\end{enumerate}
}
\end{remark}




  \end{document}